\documentclass[12pt, reqno]{amsart}

\makeatletter
\def\subsection{\@startsection{subsection}{2}%
  \z@{.5\linespacing\@plus.7\linespacing}
{.5\baselineskip}%
  {\normalfont\centering\scshape}%
}
\makeatother

\usepackage[dvipsnames]{xcolor}
\usepackage{latexsym,amsfonts,amssymb,epsfig,verbatim}
\usepackage{amsmath}
\usepackage{latexsym, graphics, textcomp}
\usepackage[pagebackref,breaklinks,hidelinks,unicode]{hyperref}
\usepackage{mathtools}
\usepackage{amsthm}
\usepackage{pinlabel}
\usepackage{braket}
\usepackage{tikz}
\usepackage[left=30mm,right=30mm,top=35mm,bottom=35mm]{geometry}
\usepackage{dirtytalk}

\usepackage[all,2cell,dvips]{xy}

\newcommand{\nc}{\newcommand}
\nc{\dmo}{\DeclareMathOperator}
\nc{\nt}{\newtheorem}

\nt{theorem}{Theorem}[section]
\nt{thm}[theorem]{Theorem}
\nt{lemma}[theorem]{Lemma}
\nt{prop}[theorem]{Proposition}
\nt{question}[theorem]{Question}
\nt{cor}[theorem]{Corollary}
\nt{fact}[theorem]{Fact}
\nt{problem}[theorem]{Problem}
\nt{conjecture}[theorem]{Conjecture}
\nt{Definition}[theorem]{Definition}
\nt{Claim}{Claim}
\nt*{claim}{Claim}
\theoremstyle{definition}
\nt*{remark}{Remark}

\nc{\Z}{\mathbb{Z}}
\nc{\R}{\mathbb{R}}
\nc{\Q}{\mathbb{Q}}
\nc{\N}{\mathbb{N}}

\nc{\margin}[1]{\marginpar{\tiny #1}}

\nc{\p}[1]{\smallskip\noindent{{\bf #1}}}

\begin{document}

\title{Coarse Kernels of Group Actions}

\author{Tejas Mittal}

\begin{abstract}
In this paper, we study the coarse kernel of a group action, namely the normal subgroup of elements that translate every point by a uniformly bounded amount. We give a complete algebraic characterization of this object. We specialize to $\mathrm{CAT}(0)$ spaces and show that the coarse kernel must be virtually abelian, characterizing when it is finite or cyclic in terms of the curtain model. As an application, we characterize the relation between the coarse kernels of the action on a $\mathrm{CAT}(0)$ space and the induced action on its curtain model. Along the way, we study weakly acylindrical actions on quasi-lines.
\end{abstract}

\maketitle
\section{Introduction}
The goal of this paper is to study certain stabilizers of group actions. Within the framework of groups acting on metric spaces, one is often interested in the properties that depend solely on the quasi-isometry type of the action. In other words, the object of interest must remain invariant under $G$-equivariant quasi-isometries. Consequently, the kernels of actions are not ideal objects of investigation. Instead, it is more appropriate to consider their coarse versions, which possess this stability feature. In the current work, we study the \emph{coarse kernel} of a group action, namely the subgroup consisting of elements that translate every point by a uniformly bounded amount.

\begin{Definition}
\label{defG_X}
Let $G$ be a group acting on a metric space $X$. Define the \textbf{coarse kernel} of the action to be ${G_X := \{g\in G
   \mid  \exists C>0 \text{ such that }d(g\cdot x,x) \leq C \text{ }\forall x\in X \}.}$\end{Definition}
   
 This coarse kernel then serves as a suitable analog to the kernel of a group action. We remark that our coarse kernels are different from the ones of \cite{Leitner2023}, where the focus is on the analogue of kernels for coarse groups. Our first result is to show that, quite surprisingly, for a geometric action, the coarse kernel $G_X$ can be described in purely algebraic terms. In the theorem below, $Z(H)$ denotes the center of the group $H$.
  \begin{theorem}
\label{corthmA}
    Let $G$ be a group acting properly discontinuously and co-boundedly on a metric space $X$. Then, $G_X = \bigcup \{Z(H) \mid |G/H|   <\infty \}$. 
\end{theorem}
      
   In particular, if $G$ is abelian, then $G_X= G$, which can be shown to hold for all co-bounded isometric actions. Similarly, in the presence of negative curvature, for instance, when a hyperbolic group acts on its Cayley graph, one expects the coarse kernel to be finite. We focus on $\mathrm{CAT}(0)$ spaces, which interpolate between the two cases above, and show the following.
\begin{theorem}
\label{ThmC}
Let $X$ be a co-bounded $\mathrm{CAT}(0)$ space, and let $G$ act on $X$ properly discontinuously.  Suppose that $G$ has an unbounded orbit in the curtain model $X_D.$ Then, $G_X$ is virtually cyclic.
Moreover, if an orbit of $G$ in $X_D$ is not a quasi-line either, then $G_X$ is finite.
\end{theorem}
The \emph{curtain model} was introduced by Petyt, Spriano, and Zalloum in \cite{PETYT2024109742}, and is a hyperbolic space $X_D$ associated to a $\mathrm{CAT}(0)$ space $X$ that \say{encodes the hyperbolicity of $X$}. This allows us to make precise the intuition that a group acting on a $\mathrm{CAT}(0)$ space that is hyperbolic enough needs to have a small coarse kernel. We remark that $G$ does not need to act co-compactly on $X$. When we add that assumption, we obtain a significantly stronger characterization of the coarse kernel. Specifically, we are able to identify the algebraic structure of the coarse kernel $G_X$ based on certain geometric constrains on $X$.
\begin{theorem}
\label{ThmC'}
Let $X$ be a $\mathrm{CAT}(0)$ space, and let $G$ act on $ X$ geometrically. Then,\\
    (i) $G_X$ is virtually $\Z^n$ for some $n \in \Z_{\geq 0}$.\\
    (ii) If moreover $X_D$ is unbounded, then $G_X$ is virtually cyclic.\\
    (iii) If furthermore $X_D$ is not a quasi-line, then $G_X$ is the largest finite normal subgroup of $G$.
\end{theorem}
After the first draft of this paper was completed, it was pointed out to us that item \emph{(iii)} also follows from \cite[Corollary~1.6]{baik2024kernel}. However, our proof and their proof uses different techniques.

Our second line of investigation is concerned more directly on the curtain model, and follows on the program started in \cite{PETYT2024109742} of better understanding the group $G$ via its action on the curtain model. A natural question to ask is how much the action of $G$ on $X_D$ might forget information. We show that, from the perspective of coarse kernels, the action of $G$ on $X_D$ retains as much information as the action of $G$ on $X$.

\begin{theorem}
\label{ThmB}
Let $X$ be a proper $\mathrm{CAT}(0)$ space with $X_D$ being unbounded, and let $G$ act on $X$ co-compactly. Then, $G_X=G_{X_D}$.
\end{theorem}

Further, for a group acting geometrically the result above gives a complete characterization of when $G_X$ and $G_{X_D}$ coincide. 

\begin{cor}
\label{corthmB}  Let $G$ be a group acting geometrically on a $\mathrm{CAT}(0)$ space $X$. Then, exactly one of the following holds.
    \begin{enumerate}
        \item $X_D$ is unbounded and hence $G_X = G_{X_D}$;
        \item $X_D$ is bounded, $|G:Z(G)|<\infty$, and $G_X = G = G_{X_D}$; 
        \item $X_D$ is bounded, $Z(G)$ has infinite index in $G$ and $G_X \neq G_{X_D}$.
    \end{enumerate}
\end{cor}

\subsection{Weak acylindricity}

In order to prove Theorem~\ref{ThmC'}, we will need to study weakly acylindrical actions on a quasi-line which is a topic of independent interest. We recall the definition of weak acylindrical action. 

\begin{Definition}
    Let $G$ be a group acting on a metric space $X$. The action of $G$ is said to be \emph{weakly acylindrical} if for each $\epsilon>0$ there exists $R$ such that
for any $x,y\in X$ with $d(x,y) > R$, only finitely many $g\in G$ satisfy $max\{d(x,g\cdot x),d(y,g\cdot y)\} < \epsilon$.
\end{Definition}

In this setting, we extend a well-known result for acylindrical actions on a quasi-line.

\begin{theorem}
Let $X$ be a roughly geodesic hyperbolic space. Let $G $ act on $ X$ weakly acylindrically with a quasi-line orbit. Then $G$ is virtually $\Z$.
\end{theorem}

As a corollary, we establish the following result, which can be viewed as a counterpart to the ping-pong lemmas for $\mathrm{CAT}(0)$ spaces, thereby helping us better understand the curtain model towards the \emph{Tits Alternative}.

\begin{theorem}\label{thmi:quasi-line}
      Let $G$ be a group acting geometrically on a $\mathrm{CAT}(0)$ space $X$. Let $H$ be any subgroup of $G$. If $|\partial_{X_D} H|=2$, then $H$ is virtually $\Z$.
\end{theorem}

\subsection{Coarse stabilizers of sets}

We, in fact, develop the proof of Theorem~\ref{corthmA} in the broader context of coarse stabilizers of subsets $Y\subseteq X$, which generalise the coarse kernel $G_X$.

\begin{Definition}
   Let $G$ be a group acting on a metric space $X$ and let $Y \subseteq X$. Define the \textbf{coarse stabilizer of Y} to be $G_Y := \{g\in G
   \mid  \exists C \text{ such that }d(g\cdot y,y) \leq C \text{ }\forall y\in Y \}$.
\end{Definition}

The coarse stabilizers are a very well-studied object in geometric group theory. Typically, the set $Y$ is unbounded; otherwise $G_Y = G$. In this sense, one can see that by varying $Y$, one interpolates between $G$ and the coarse kernel $G_X$. We now give a complete algebraic characterization of the coarse stabilizer $G_Y$ when $Y$ coarsely coincides with an orbit of a finitely generated subgroup $H$ of $G$. 
In particular, this shows that when $G$ acts geometrically on $X$, we recover Theorem~\ref{corthmA}. In the theorem below, $C_G(H)$ denotes the centralizer of $H$ in $G$.
 
\begin{theorem}
\label{ThmA}
Let $X$ be a metric space, and suppose $G $ acts on $ X$ properly discontinuously and co-boundedly. Let $H\leq G$ be finitely generated, and let $Y\subseteq X$ coarsely coincide with an orbit of $H$. Then, $G_Y = \bigcup \{C_G(H') \mid |H/H'|   <\infty \}$.
\end{theorem}

\subsection*{Acknowledgement}
This work was done as part of an undergraduate research project under the supervision of Davide Spriano. Davide suggested the initial problems which evolved into this paper, and his insights helped me resolve many queries. His dedicated time also helped me acquire the necessary background knowledge and assisted in refining several of the proofs. I am extremely grateful to Davide for taking the time to supervise this project and would like to thank him for all his guidance and support. I would also like to thank Harry Petyt for his kind help, which enhanced my understanding of the topic, and for his helpful input on an earlier draft of the paper.

\section{Background}

In this section, we present definitions and established results necessary to prove our main theorems later. The references for this material are \cite{PETYT2024109742} and \cite{caprace:amenable}.

\subsection{Curtain Model}
\label{sec2.1}
We briefly discuss the \emph{curtain model} $X_D$ associated with any $\mathrm{CAT}(0)$ space $X$, as defined in \cite{PETYT2024109742}. Essentially, the curtain model $X_D:=(X,D)$ is the set $X$ equipped with a new metric $D:X\times X\rightarrow  \R$. We aim to choose the metric $D$ so that $X_D$ becomes a hyperbolic space.\\
The main ingredient in defining $X_D$ is the notions of \emph{curtains} and that of \emph{separation}.

\begin{Definition}[Definition 2.1 in \cite{PETYT2024109742}]
Let $X$ be a $\mathrm{CAT}(0)$ space, and let $\alpha:I \rightarrow X$ be a geodesic. For a number $r$ with $[r-1/2,r+1/2]$ in the interior of $I$, the \emph{curtain} dual to $\alpha$ at $r$ is 
   $$ h_{\alpha,r}=\pi_\alpha^{-1} (\alpha[r-1/2,r+1/2]).$$
Here $\pi_\alpha$ denotes the closest-point projection onto $\alpha$.
\end{Definition}

\begin{Definition}[Definition 2.2 in \cite{PETYT2024109742}]
Let $X$ be a $\mathrm{CAT}(0)$ space, and let $h=h_{\alpha,r}$ be a curtain. The \emph{halfspaces} determined by $h$ are $h^{-}=\pi_\alpha^{-1} \alpha(I\cup (-\infty, r-1/2))$ and $h^{+}=\pi_\alpha^{-1} \alpha(I\cup (r+1/2,\infty))$. Note that $\{h^{-},h,h^{+}\}$ is a partition of $X$. If $A$ and $B$ are subsets of $X$ such that $A\subseteq h^-$ and $B\subseteq h^+$, then we say that $h$ separates $A$ from $B$.
\end{Definition}

Next, we define \emph{chains} and the \emph{chain distance}.
\begin{Definition}[Definition 2.9 in \cite{PETYT2024109742}]
     A set $\{h_i\}$ of curtains is a \emph{chain} if $h_i$ separates $h_{i-1}$ from $h_{i+1}$ for all $i$. We say that $\{h_i\}$ separates $A, B \subseteq X$ if every $h_i$ does.
     The \emph{chain distance} from $x$ to $y\neq x$ is $d_{\infty}(x,y)=1+ max\{|c| : \text{c is a chain separating x, y}\}$.
\end{Definition}

This allows us to define the notion of an \emph{$L$-chain}.
\begin{Definition}[Definition 2.11 in \cite{PETYT2024109742}]
    Let $L\in \N$. Disjoint curtains $h$ and $h'$ are said to be \emph{$L$-separated} if every chain meeting both $h$ and $h'$ has cardinality at most $L$. If $c$ is a chain of curtains such that each pair is $L$-separated, then we refer to $c$ as an \emph{$L$-chain}.
\end{Definition} 

Now, we use the notion of $L$-chains to define a family of metric spaces $X_L:=(X,d_L)$ corresponding to $L\in \N$.
\begin{Definition}[Definition 2.15 in \cite{PETYT2024109742}]
    Given distinct points $x, y \in X$, set $d_L(x, x)=0$ and define $d_L(x,y)=1+ max\{|c| : \text{c is an L–chain separating x from y}\}$.
\end{Definition}

\begin{remark}
One can check that $d_L:X \times X\rightarrow  \R$ is indeed a metric on $X$ and that $d_L(x,y) \leq d_{\infty}(x,y) \leq 1+d(x,y)$ for all $L\in \N$.
\end{remark}

Finally, we define the \emph{curtain model} $X_D$ by combining all the $L$-metrics $d_L$ on $X$.
\begin{Definition}
    Fix a sequence $\lambda_L\in (0,1)$ such that \[ \sum_{n=1}^{\infty}  \lambda_L < \sum_{n=1}^{\infty}  L\lambda_L < \sum_{n=1}^{\infty}  L^2\lambda_L = \Lambda<\infty. \] Then, $X_D:=(X,D)$ where $D(x,y):= \sum_{L=1}^{\infty} \lambda_L d_L(x,y)$ for any $x,y\in X$.
\end{Definition}

We record the following basic fact, which tells us how distances in the curtain model $X_D$ are bounded in terms of the distances in $X$.
\begin{lemma}
 \label{bounddistance}
 $D(x,y)\leq \Lambda (1+d(x,y))$ for all $x,y\in X$.
\end{lemma}
\begin{proof}
    $D(x,y):= \sum_{L=1}^{\infty} \lambda_L d_L(x,y) \leq \sum_{L=1}^{\infty} \lambda_L (1+d(x,y)) < \Lambda (1+d(x,y))$.
\end{proof}

Now, we state some properties of the \emph{curtain model}.

Firstly, we note that $X_D$ is \emph{$\delta$-hyperbolic} in the sense of Gromov's four point condition, for some $\delta>0$. We also examine how the properness of a group action on a $\mathrm{CAT}(0)$ space $X$ is revised when we move to the curtain model $X_D$.\\
This is summarised in Theorem \ref{9.51015} below, which follows directly from Proposition 9.5, Theorem 9.10, and Proposition 9.16 of \cite{PETYT2024109742}.

\begin{theorem}
    \label{9.51015}
    Let $X$ be a $\mathrm{CAT}(0)$ space. Then, $X_D$ is a roughly geodesic hyperbolic space. Moreover, if $G$ is a group acting properly discontinuously on $X$, then the induced action of $G$ on $X_D$ is weakly acylindrical.
\end{theorem}
Now, we look at how the Gromov boundary $\partial X_D$ of $X_D$ embeds in the visual boundary $\partial X$ of $X$. This will be needed in section 5 to prove Theorem \ref{ThmB}.
\begin{thm}[Theorem L in \cite{PETYT2024109742}]
 \label{theoremL}
 Let $X$ be a proper $\mathrm{CAT}(0)$ space. Then the space $\partial X_D$ embeds homeomorphically as an $\mathrm{Isom}$ $X$-invariant subspace of $\partial X$, and every point in the image of $\partial X_D$ is a visibility point of $\partial X$. The embedding is induced by the change-of-metric map $X_D\rightarrow X$. Moreover, if $\partial X_D$ is nonempty and $X$ is cobounded, then $\partial X_D$ is dense in $\partial X$.
\end{thm}
\subsection{Actions on a Hyperbolic Space}
\label{sec2.2}
Here, we briefly review Gromov's classification of group actions on hyperbolic spaces as given in Section 3.A of \cite{caprace:amenable}. We will use this in Section 5 to prove Theorem \ref{ThmC}.

Firstly, we recall Gromov's classification of isometries of hyperbolic spaces in terms of the translation length.

\begin{Definition}
    Let $X$ be a metric space and $\phi$ be an isometry of $X$. Then, the \emph{translation length} of $\phi$ is $\tau(\phi):=\lim_{n\rightarrow \infty} d(x,\phi^n(x))/n$.
\end{Definition}

\begin{remark}
    One can check that $\tau(\phi)$ is well-defined and is independent of the choice of $x\in X$.
\end{remark}

Let $X$ be a hyperbolic space. An isometry $\phi$ of $X$ is called -\\
      \textbullet  \text{} \emph{Elliptic} if $\phi$ has bounded orbits.\\
      \textbullet \text{} \emph{Parabolic} if $\phi$ has unbounded orbits and $\tau(\phi)=0$.\\
      \textbullet \text{} \emph{Hyperbolic} if $\tau(\phi)>0$.

Next, we recall the notion of the \emph{Gromov boundary} of a hyperbolic space $X$ and that of the \emph{limit set} of a group in $X$.

Fix a base point $x\in X$. Define the \emph{Gromov product} of points $y,z\in X$ with respect to $x$ as $(y|z)_x:= (d(x,y)+d(x,z)-d(y,z))/2$.
A sequence $(x_n)$ in $X$ is said to be \emph{Cauchy-Gromov} if $(x_n|x_m)_x\rightarrow \infty$ as $n,m \rightarrow \infty$. For convenience, we will write $(x_n|x_m)$ to mean $(x_n|x_m)_x$.

\begin{remark}
Note that $|(y|z)_x - (y|z)_w| = |d(x,y)-d(w,y)+d(x,z)-d(w,z)| /2 \leq d(x, w)$. So the notion of a Cauchy-Gromov sequence is independent of the base point $x\in X$.
\end{remark}

Define an equivalence relation on the Cauchy-Gromov sequences as follows: two sequences $(y_n)$ and $(z_n)$ are equivalent, denoted
$(y_n)\sim (z_n)$, if $(y_n|z_n)\rightarrow \infty$ as $n\rightarrow \infty$.

\begin{Definition}
    Let $X$ be a hyperbolic space. The \emph{gromov boundary} $\partial X$ of $X$ is $\partial X:= \{(x_n) \mid (x_n) \text{ is a cauchy-gromov sequence}\}/{\sim}$.
\end{Definition}

\begin{Definition}
    Let $G$ be a group acting on a hyperbolic space $X$ via isometries. The \emph{limit set} of $G$ in $X$ is {$\partial_X(G):= \{(y_n)\in \partial X \mid (y_n) \sim (g_n\cdot x) \text{ for some } x\in X, g_n\in G\}$}.
\end{Definition}

Now, we state Gromov's classification of group actions on hyperbolic spaces.

\begin{theorem}[Gromov's Classification]
  \label{classification1} 
Let $G$ be a group acting on a hyperbolic geodesic metric space $X$. Then exactly one of the following holds, and the action of $G$ is said to be -\\
    (i) \textbf{bounded} if the orbits of $G$ in $X$ are bounded.\\
    (ii) \textbf{horocyclic} if the orbits are unbounded and $G$ contains no hyperbolic isometry.\\
    (iii) \textbf{lineal} if $G$ contains a hyperbolic isometry, and any two hyperbolic elements have the same endpoints.\\
    (iv) \textbf{focal }if it is not lineal, $G$ contains a hyperbolic isometry, and all hyperbolic elements have one common endpoint.\\
    (v) \textbf{general type} if $G$ contains two hyperbolic isometries which have no common end-point.
\end{theorem}

In fact, the above classification can be described in terms of the limit set $\partial_X (G)$.

\begin{theorem}[Proposition 3.1 in \cite{caprace:amenable}]
\label{classification2}
    Let $G$ be a group acting on a hyperbolic geodesic metric space $X$. Then, the action of $G$ is -\\
  (i) bounded $\iff \partial_X(G)$ is empty.\\
  (ii) horocyclic $\iff |\partial_X(G)|=1$; then $\partial_X(G)$ is the unique finite orbit of $G$ in $\partial X$.\\
  (iii) lineal $\iff |\partial_X(G)|=2$; then $\partial_X(G)$ contains all the finite orbits of $G$ in $\partial X$.\\
  (iv) focal $\iff \partial_X(G)$ is uncountable and $G$ has a fixed point $\xi$ in $\partial_X(G)$; then $\xi$ is the unique finite orbit of $G$ in $\partial X$.\\
  (v) general type $\iff \partial_X(G)$ is uncountable and $G$ has no finite orbit in $\partial X$.
\end{theorem}
In proving Theorem \ref{ThmC} using this classification, it will be important for us to be able to rule out \emph{horocyclic} and \emph{lineal} actions. We record two lemmas regarding this.

\begin{lemma}
    \label{linealactions}
   Let $X$ be a hyperbolic geodesic metric space and suppose the action of $G$ on $X$ is lineal. Then $G$ has a quasi-line orbit in $X$. 
\end{lemma}
\begin{proof}
This is well-known, but we include a proof for completeness.\\
Let $\{\xi^+, \xi^-\}$ be the limit set of $G$. 
Since the Gromov boundary is a visibility space, there exists a geodesic $\gamma$ with endpoints $\xi^-$ and $\xi^+$. Since $G$ fixes $\partial_X (G)$, it fixes the endpoints of $\gamma$. Thus, $g\cdot \gamma$ is a geodesic with the same endpoints as $\gamma$ for every $g\in G$. By hyperbolicity of $X$, the Hausdorff distance between $\gamma$ and $g\cdot \gamma$ is uniformly bounded. In particular, there exists $R$ such that $g\cdot \gamma \subseteq N_R(\gamma)$ for all $g\in G$. \\
Now, fix a basepoint $x_0$ on $\gamma$. We will show that the orbit of $x_0$ is a quasi-line. The previous argument yields that $G \cdot x_0 \subseteq N_R(\gamma)$. Moreover, $G$ contains a hyperbolic isometry by Theorem~\ref{classification1}, and hence we get that the orbit $G \cdot x_0$ coarsely coincides with a neighbourhood of $\gamma$.
\end{proof}
\begin{lemma}
    \label{horocyclicactions}
    Let $G$ be a group acting co-boundedly on a hyperbolic geodesic metric space $X$. Then the action of $G$ on $X$ is not horocyclic.
\end{lemma}
\begin{proof}
    Since the action of $G$ on $X$ is co-bounded, the orbit $G \cdot x_0$ is quasi-dense in $X$. Consequently, the orbit $G \cdot x_0$ is quasi-convex. By Proposition 3.2 of \cite{caprace:amenable}, it follows that the action of $G$ cannot be horocyclic.
\end{proof}
\begin{remark}
    Note that Theorems \ref{classification1}, \ref{classification2}, and Lemmas \ref{linealactions}, \ref{horocyclicactions} also hold for roughly geodesic hyperbolic spaces $X$. This is because we may replace $X$ by the injective hull $E(X)$ of $X$ which now becomes a geodesic hyperbolic space. Moreover, there is an isometric embedding $i: X \rightarrow E(X)$ which is coarsely surjective and $G$-equivariant.\end{remark}

\section{Coarse stabilizer and the coarse fixed set}
 
We work with a general metric space $X$ and give an algebraic characterisation of the coarse stabilizer $G_Y$ of $Y$.
In particular, we show that if $G$ acts geometrically on $X$ and $Y$ coarsely coincides with an orbit of a finitely generated subgroup $H$ of $G$, then $G_Y$ is the union of centralisers $C_G(H')$ where $H'$ runs over finite index subgroups of $H$.\\
In Section 5, we use this to show that for a $\mathrm{CAT}(0)$ group, $G_X$ is virtually $\Z^n$.

The difficult part of the theorem lies in showing that if $g\in G_Y$, then $g\in C_G(H')$ for some finite index subgroup $H'$ of $H$. \\
For this, our strategy will be to show that $C_G(\langle g \rangle)$ acts geometrically on a subset $CFix_X(\langle g \rangle)$ of $X$ containing $Y$. Then, it will follow that $C_G(\langle g \rangle)$ must contain a finite index subgroup of $H$, giving us the desired result.

To this end, we define the following.

\begin{Definition}
   Let $G$ be a group acting on a metric space $X$ and let $H\leq G$. Define the \textbf{coarse fixed set of H} to be $CFix_X(H) := \{x\in X \mid  d(h\cdot x,x)\leq C(h) \text{ } \forall h\in H\}$. Here $C(h)$ is any constant depending on $h$, chosen so that $CFix_X(H) \neq \emptyset$. \\
   (Note that we can always find such a function $C: H \rightarrow \R$. Indeed, fix $x_0 \in X$, then defining $C(h) := d(x_0,h\cdot x_0)$ ensures that $x_0\in CFix_X(H)$.)
\end{Definition}

A consequence of Theorem \ref{Theorem A} below is that if $G$ is a group acting geometrically on $X$ and  $H$ is a finitely generated subgroup of $G$  then, up to quasi-isometries, $CFix_X(H)$ does not depend on the choice of the function $C: H\rightarrow \R$. 

\begin{theorem}
\label{Theorem A}
Let $X$ be a metric space and suppose $G $ acts on $ X$ properly discontinuously and co-boundedly. Let $H\leq G$ be finitely generated, and let $Y\subseteq X$ coarsely coincide with an orbit of $H$. Then, \\
(i) $C_G(H)$ acts geometrically on $ CFix_X(H) $\\
(ii) $G_Y = \bigcup \{C_G(H') \mid |H/H'|   <\infty \}$.
\end{theorem}

\begin{proof}
 (i) This can be proven in a similar way to Theorem 3.2 in  \cite{Ruane:actiononboundary}. \\
 We include the proof for completeness.\\
 Since $H$ is finitely generated, so let $h_1, h_2, ..., h_n$ generate $H$. It is straightforward to see that $C_G(H)\cdot CFix_X(H) \subset CFix_X(H)$. Since $G$ acts on $X$ properly discontinuously, we have that $C_G(H)$ acts on $CFix_X(H)$ properly discontinuously. \\
 So it remains to prove that $C_G(H)$ acts on $CFix_X(H)$ co-boundedly. Let $x_0\in CFix_X(H)$ and suppose for a contradiction that $C_G(H)$ does not act co-boundedly on $CFix_X(H)$. Then there exists $z_n \in CFix_X(H)$ with $d(z_n,C_G(H)\cdot x_0)\rightarrow \infty$ as $n\rightarrow \infty$.  And since $G\cdot x_0$ is quasi-dense in $X$, so there exists $g_n\in G$ and $L >0$ so that $d(g_n \cdot x_0, z_n)<L$ for all $n\in \N$.  Hence, we have 
 \begin{equation}
\begin{split}
\label{e1}
d(g_n\cdot x_0,C_G(H)\cdot x_0)& \geq d(z_n,C_G(H)\cdot x_0)-d(g_n\cdot x_0,z_n)\\ & \geq d(z_n,C_G(H)\cdot x_0)-L\\ &\rightarrow \infty \text{ as } n\rightarrow \infty.
 \end{split}
 \end{equation}
Let $K = max_{i= 1,2,..,n}C(h_i)+2L$. So,   
$$d(x_0,g_m^{-1}h_ig_m\cdot x_0) = d(g_m\cdot x_0,h_ig_m\cdot x_0) \leq d(z_m,h_i\cdot z_m)+2L \leq C(h_i)+2L \leq K.$$
Since $G$ acts on $X$ properly discontinuously, so there are only finitely many $g\in G$ with $g\cdot x_0 \in B_K(x_0)$. Thus,  $\{g_m^{-1}h_ig_m : m\in \N, 1\leq i\leq n\}$ is a finite set.

So, we may restrict to a subsequence of $(g_n)$ to get that the equation ${g_m^{-1}h_1g_m = g_n^{-1}h_1g_n}$ holds for all $g_m,g_n$ in the subsequence. Doing the same for $h_2,h_3,..,h_n$ we get a further subsequence $g_{m_n}$ of $(g_n)_{n\geq 1}$ so that $g_{m_{n'}}^{-1}h_ig_{m_{n'}} = g_{m_n}^{-1}h_ig_{m_n}$ for all $ i,n,n'$. 
Thus, ${g_{m_n}g_{m_1}^{-1}\in C_G(H)}$ for all $n\in \N$ and a fixed $g_{m_1}$. \\
So $d(g_{m_n}\cdot x_0,C_G(H)\cdot x_0) \leq d(g_{m_n}\cdot x_0,g_{m_n}g_{m_1}^{-1}\cdot x_0) = d( x_0,g_{m_1}^{-1}\cdot x_0)$ is bounded as $n$ tends to infinity, contradicting equation \ref{e1}.
 
 (ii) If $g\in G_Y$, then $Y \subseteq CFix_X(\langle g \rangle)$. 
 Note that $CFix_X(\langle g \rangle)$ depends on a function $f \colon \langle g \rangle  \rightarrow \R$.  But we can choose $f$ to satisfy $f(g^n)= nC$ where $d(g\cdot y, y) \leq C$ for all $y\in Y$. This ensures $Y \subseteq CFix_X(\langle g \rangle)$. \\
 By (i), $C_G(g)$ acts geometrically on $CFix_X(\langle g \rangle)$. Thus, if $y_0\in Y$, then ${C_G(g)\cdot y_0 \cap N_R(Y)}$ is quasi-dense in $N_R(Y)$ for some neighborhood $N_R(Y)$ of $Y$. Since $Y$ coarsely coincides with some (and hence every) orbit of $H$, the orbit $H\cdot y_0$ is quasi-dense in some neighborhood $N_{R'}(Y)$ of $Y$. Therefore, there is a constant $D$ such that for all $ h\in H$, there exists $ z\in C_G(g)$ with $d(h\cdot y_0,z\cdot y_0) < D$. In other words, $d(z^{-1} h\cdot y_0,y_0) < D$,  and thus $z^{-1}h \in \{ g_1,g_2,...,g_n \}$, a finite set. This is because the action of $G$ on $X$ is properly discontinuous, meaning that there are only finitely many $g\in G$ with $g\cdot y_0 \in B_D(y_0)$. \\
 So, $H \subseteq C_G(g)\cdot g_1 \cup C_G(g)\cdot g_2 \cup ... \cup C_G(g)\cdot g_n$. Thus, $[C_G(g)H:C_G(g)] \leq n$ and hence $[H:C_G(g)\cap H] = [C_G(g)H:C_G(g)]\leq n$. The last equation is using the fact that for two subgroups $A,B$ it holds $[A:A\cap B] = [BA : B]$, where $BA$ is the set $\{ba \mid b\in B, a\in A\}$. Thus, $H':= C_G(g)\cap H$ is a finite index subgroup of $H$. Also, $g\in C_G(H')$ because $H'\subseteq C_G(g)$ and so $g$ commutes with all $h'\in H'$. \\
 Hence, $g\in \bigcup \{C_G(H') \mid |H/H'|   <\infty \}$.
 
 Conversely, if $g \in C_G(H')$ for some finite index subgroup $H'$ of $H$, then we would have $d(g\cdot (h\cdot y_0), h\cdot y_0) = d(h\cdot (g\cdot y_0), h\cdot y_0) = d(g\cdot y_0, y_0)$ for all $h\in H'$. Therefore, $g\in G_{H'\cdot y_0}$. Since $H'\cdot y_0$ is quasi-dense in $Y$ (as $H$ acts co-boundedly on $Y$, and $|H/H'|<\infty$), it follows that $g\in G_Y$.
 Indeed, let $C>0$ be a constant such that for all $y\in Y$, there exists $h\in H'$ so that $d(h\cdot y_0,y)\leq C$. Then, for all $y\in Y$, we have $$d(g\cdot y, y) \leq d(g\cdot (h\cdot y_0),g\cdot y)+d(g\cdot (h\cdot y_0), h\cdot y_0)+d(h\cdot y_0, y) \leq 2C+d(g\cdot y_0, y_0).$$
\end{proof}

\begin{remark}
    In particular, for a finitely generated subgroup $H\leq G$ we can tell that $C_G(H)$ is infinite just by showing $CFix_X(H)$ is unbounded for some metric space $X$ on which $G$ acts geometrically. 
\end{remark}
\begin{remark}
   Note that under the assumptions as above, we also have that, up to quasi-isometries, $CFix_X(H)$ is independent of the choice of the function $C: H \rightarrow \R$.\\
   Indeed, suppose $CFix_X(H,C_{1})$ and $CFix_X(H,C_2)$ are two such subsets of $X$ corresponding to functions $C_1,C_2:H\rightarrow \R$. Pick any $x_0\in CFix_X(H,C_{1})$. Part (i) of Theorem \ref{Theorem A} tells us that $C_G(H)\cdot x_0$ is quasi-dense in $CFix_X(H,C)$ for any function $C$. And since $x_0\in CFix_X(H,C_{1})\subseteq CFix_X(H,C_1+C_2)$, we have that $C_G(H)\cdot x_0$ is quasi-dense in $CFix_X(H,C_1+C_2)$. If we denote by $A \sim B$ the equivalence relation of quasi-isometry, then we have $CFix_X(H,C_1+C_2) \sim C_G(H)\cdot x_0 \sim CFix_X(H,C_{1})$. 
   Similarly, we get $CFix_X(H,C_1+C_2) \sim CFix_X(H,C_2)$. 
   Therefore, $CFix_X(H,C_2)\sim CFix_X(H,C_{1})$.
\end{remark} 

\begin{cor}
\label{cortheoremA}
    Let $G$ be a group acting properly discontinuously and co-boundedly on a metric space $X$. Then, $G_X = \bigcup \{Z(H) \mid |G/H|   <\infty \}$.
\end{cor}

\begin{proof}
    Applying Theorem \ref{Theorem A} with $Y=X$ gives us $G_X = \bigcup \{C_G(H) \mid |G/H|   <\infty \}$. Now, $g\in Z(H)$ implies $g\in C_G(H)$, and hence $\bigcup \{Z(H) \mid |G/H|   <\infty \} \subseteq G_X$. Conversely, if $g\in G_X$, then $g\in C_G(H)$ for some finite index subgroup $H$ of $G$. Letting $\Tilde{H}$ be the subgroup of $G$ generated by $H\cup\{g\}$, we see that $g\in Z(\Tilde{H})$. Moreover, since $H\subseteq \Tilde{H} \subseteq G$ and $|G/H|<\infty$, $\Tilde{H}$ is a finite index subgroup of $G$.\\
    Therefore, $g\in \bigcup \{Z(H) \mid |G/H|   <\infty \}$.
\end{proof}

\section{Weakly-acylindrical actions}

This is an aside on weakly acylindrical actions on a quasi-line. We show that a group acting weakly acylindrically on a roughly geodesic hyperbolic space with a quasi-line orbit has to be virtually $\Z$. This will be usefull later while proving Theorem \ref{Theorem C}. 

\begin{prop}
\label{Theorem D} 
Let $X$ be a roughly geodesic hyperbolic space. Let $G$ act on $X$ weakly acylindrically with a quasi-line orbit. Then $G$ is virtually $\Z$.
\end{prop}
\begin{proof}
Without loss of generality, we may assume that $X$ is a geodesic metric space. This is because we may replace $X$ by $E(X)$, the injective hull of $X$. Then, there exists $i:X\rightarrow E(X)$, an isometric embedding, which is coarsely surjective and $G$-equivariant. It is straightforward to check that $G$ acts weakly acylindrically on $E(X)$. Also, if $G\cdot x_0$ is a quasi-line orbit of $G$ in $X$ then $i(G\cdot x_0)$ is a quasi-line orbit of $G$ acting on $E(X)$.

So, let $X$ be a geodesic hyperbolic space and $Y\subseteq X$ be a quasi-line orbit of $G$. 

Note that the action of $G$ on $Y$ extends to an action of $G$ on  $\partial Y$. Thus, we get a homomorphism $\phi: G\rightarrow \operatorname{Homeo}(\partial Y)$ with kernel  $K$. Since $|\partial Y|= 2$, we have $|G/K|\leq |\operatorname{Homeo}(\partial Y)|=2$. Thus, $K$ has finite index in $G$. So, an orbit of $K$ is quasi-dense in $Y$, and hence is a quasi-line. Thus, it suffices to prove the proposition for $K$. It will then follow that $K$, and consequently $G$, is virtually $\Z$. 
So, we assume, without loss of generality, that the action of $G$ on $\partial Y$ is trivial.

Now, $|\partial_{Y} G| =|\partial Y|= 2$ and hence, by Theorem \ref{classification2}, it follows that the action of $G$ on $Y$ is lineal. So, $G$ contains a hyperbolic element $g\in G$ i.e. there exists $g\in G$ satisfying  $\tau := \lim_{n\to\infty} d(x,g^n\cdot x)/n > 0$. Since $Y$ is a quasi-line orbit of $G$, we get that for any $y\in Y$, the orbit $\langle g \rangle \cdot y$ is quasi-dense in $Y$.
So, $\psi : \Z \rightarrow Y$ given by $\psi (n)=g^n\cdot y$, is a $(L,C)-$quasi-isometry for some $C,L>0$.

Now, we will show that $\langle g \rangle$ has finitely many right cosets in $G$. This implies that $G$ is virtually isomorphic to $\langle g \rangle$ which is isomorphic to $\Z$.

Fix $y\in Y$. Let $\langle g \rangle a$ be a right-coset of $\langle g \rangle$.\\
As $\psi$ is a $(L,C)-$quasi-isometry, so $d(a\cdot y,\psi (m))=d(a\cdot y,g^m\cdot y)<C$ for some $m\in \Z$. Let $h=g^{-m}a$. Then, 
\begin{equation}
\label{eq1}
    d(h\cdot y, y)=d(g^{-m}a\cdot y,y)=d(a\cdot y,g^m\cdot y)<C
\end{equation}
Similarly, $h g^n \cdot y$ is $C$-close to $g^{b_n}\cdot y$ for some $b_n \in \Z$. Moreover, as $h$ acts trivially on $\partial Y$, so $({g^{b_n}\cdot y})_{n\geq 0}= h\cdot ({g^n\cdot y})_{n\geq 0} = ({g^n\cdot y})_{n\geq 0} \in \partial Y$. Thus, $b_n \rightarrow \infty$ as $n \rightarrow \infty$. Hence, $b_n > 0$ for all $n \geq M$ for some $M \in \N$.

Fix $n \geq M$ and let $N > max(b_n, n)$. Now, $\psi([0,N])$ is a $(L,C)-$quasi-geodesic with endpoints $y,g^N\cdot y$. Let $\gamma$ be a geodesic from $y$ to $g^N\cdot y$. By the Morse lemma, the hausdorff distance $d_H(\gamma, \psi([0,N])) \leq K$ for some $K>0$ which depends just on $C$, $L$. Thus, there exists $c_n, d_n\in \gamma$ which satisfy $d(c_n,g^n\cdot y) \leq K$ and $d(d_n, g^{b_n}\cdot y)\leq K$. 

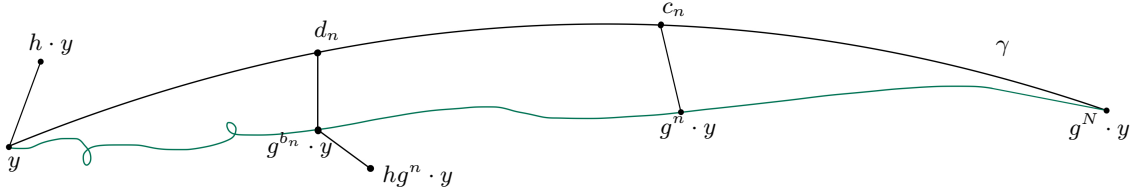
\begin{figure}
       \centering
       \includegraphics[width=3.5 cm, angle = 270]{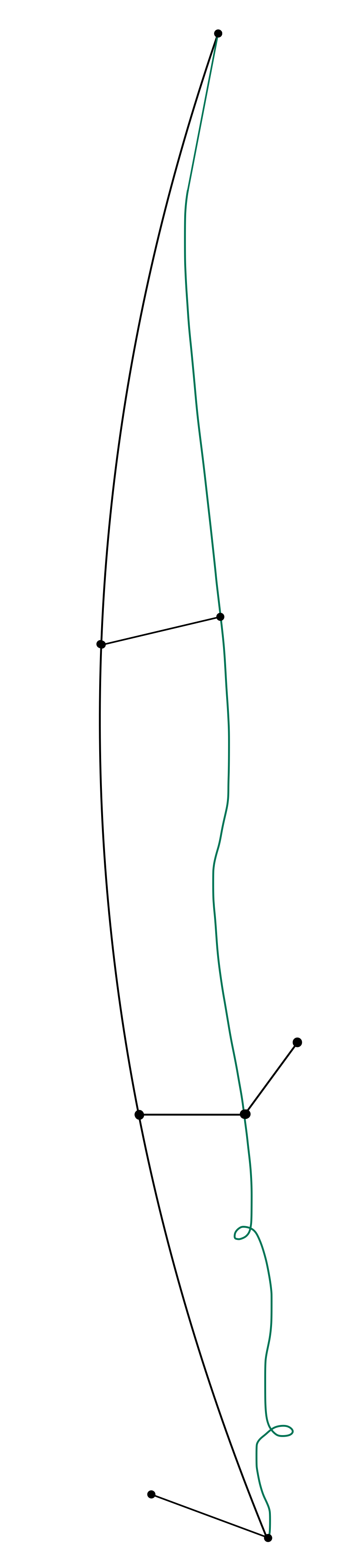}

       \begin{tikzpicture}[remember picture, overlay]
\node[anchor=north east, inner sep=0, font=\scriptsize] at (-2.9,2.62) {$d_n$};
\node[anchor=north east, inner sep=0, font=\scriptsize] at (-3,1.13) {$g^{b_n}\cdot y$};
\node[anchor=north east, inner sep=0, font=\scriptsize] at (-1.4,0.7) {$hg^n\cdot y$};
\node[anchor=north east, inner sep=0, font=\scriptsize] at (-7.15,0.8) {$y$};
\node[anchor=north east, inner sep=0, font=\scriptsize] at (-6.45,2.45) {$h\cdot y$};
\node[anchor=north east, inner sep=0, font=\scriptsize] at (7.6,1.38) {$g^N\cdot y$};
\node[anchor=north east, inner sep=0, font=\scriptsize] at (2.1,1.34) {$g^n\cdot y$};
\node[anchor=north east, inner sep=0, font=\scriptsize] at (1.7,2.85) {$c_n$};
\node[anchor=north east, inner sep=0, font=\scriptsize] at (6,2.35) {$\gamma$};

  \end{tikzpicture}
\caption{The proof of Proposition \ref{Theorem D}}     
       \label{fig:Fig}
   \end{figure}

We hence obtain the following inequalities.
\begin{equation}\label{3}
\begin{split}
|d(y, d_n) - d(y, g^n\cdot y)| &= |d(y, d_n) - d(h\cdot y,h g^n\cdot y)|\\ &\leq d(y,h\cdot y)+ d(d_n, g^{b_n}\cdot y)+ d(g^{b_n}\cdot y, h g^n \cdot y)\\ &\leq 2C+K.
\end{split}
\end{equation}
\begin{equation}\label{4}
 |d(y, c_n) - d(y, g^n\cdot y)| \leq d(c_n, g^n\cdot y) \leq K.
\end{equation}
Thus, we have
\begin{equation}\label{5}
\begin{split}
d(g^{b_n}\cdot y, g^n \cdot y) &\leq d(d_n,c_n)+d(d_n,g^{b_n}\cdot y)+d(c_n,g^n\cdot y)\\ &\leq d(c_n,d_n) + 2K \\ &= |d(c_n,y)-d(d_n,y)|+2K \\ &\leq |d(y, c_n) - d(y, g^n\cdot y)|+ |d(y, d_n) - d(y, g^n\cdot y)|+2K \\ &\leq 2C+4K.
\end{split}
\end{equation}
So, $d(h g^n \cdot y,g^n \cdot y)\leq d(g^{b_n}\cdot y, g^n \cdot y)+d(g^{b_n}\cdot y,h g^n \cdot y)\leq 3C+4K$. 

As $G$ acts weakly acylindrically on $X$, so there exists $R>0$ such that for any $y_1, y_2 \in Y$ with $d(y_1,y_2)>R$ there are only finitely many $k \in G$ which satisfy ${max(d(k\cdot y_1,y_1),d(k\cdot y_2,y_2))\leq 3C+4K}$. Since, $\langle g \rangle \cdot y$ is quasi-dense in a quasi-line, so we can choose $n$ large enough so that $d(y, g^n \cdot y)> R$. Applying the condition that $G$ acts weakly acylindrically on $Y$ with $y_1=y$, $y_2 = g^n \cdot y$, and $k=h:= g^{-m}a$, we see that there are only finitely many possible values for $g^{-m}a$. 

So, there are only finitely many right-cosets $\langle g \rangle a = \langle g \rangle g^{-m}a$ of $\langle g \rangle$.
\end{proof}

We record a consequence of the above proposition, which can be seen as part of the \emph{Tits Alternative} for $\mathrm{CAT}(0)$ groups, serving as a counterpart to the ping-pong lemmas.

\begin{cor}
    Let $G$ be a group acting geometrically on a $\mathrm{CAT}(0)$ space $X$. Let $H$ be any subgroup of $G$. If $|\partial_{X_D} H|=2$ then $H$ is virtually $\Z$.
\end{cor}
\begin{proof}
    As $G$ (and hence $H$) acts properly discontinuously on $X$, so by Theorem \ref{9.51015}, $X_D$ is a roughly geodesic hyperbolic space and $H$ acts weakly acylindrically on $X_D$. Also, since $|\partial_{X_D} H| =2$, by Theorem \ref{classification2} and Lemma \ref{linealactions}, the action of $H$ on $X_D$ is lineal and $H$ has a quasi-line orbit in $X_D$. Hence, by Proposition \ref{Theorem D}, $H$ is virtually $\Z$.
\end{proof}

\section{\texorpdfstring{$\mathrm{CAT}(0)$}{CAT(0)} Spaces}

In this section, we specialise to the case when $X$ is a $\mathrm{CAT}(0)$ space.

First, we relate the coarse kernels $G_X$ and $G_{X_D}$ corresponding to the action of $G$ on $X$ and the induced action of $G$ on $X_D$. In particular, we show that if $X$ is a proper $\mathrm{CAT}(0)$ space with an unbounded curtain model $X_D$, and $G$ acts co-compactly on $X$, then $G_X$ and $G_{X_D}$ coincide.

Our strategy will be to look at the action of $G$ on the boundaries $\partial X$ and $\partial X_D$ of $X$ and $X_D$ respectively. In particular, we look at the elements in $G$ which fix the boundaries of $X$ and $X_D$. Define, ${K_X := \{ g\in G \mid g \text{ fixes the boundary } \partial X \text{ pointwise} \}}$ and $K_D := \{ g\in G \mid g \text{  fixes } \partial X_D \text{ pointwise}\}$.
We show that $G_X \subseteq G_{X_D} \subseteq K_D$ and that   $G_X=K_X=K_D$. These together give us $G_X = G_{X_D}$.
\begin{lemma}
$\label{lemma 1}$
Let $X$ be a $\mathrm{CAT}(0)$ space and let $G$ be a group acting via isometries on $X$. If $G_X$, $G_{X_D}$, and $K_D$ are as defined above, then $G_X \subseteq G_{X_D} \subseteq K_D$.
 \end{lemma}\begin{proof}
If $g \in G_X$, then $d(x,g \cdot x)\leq C$ for all $x\in X$, for some fixed constant $C$. Hence, by Lemma \ref{bounddistance}, we have $D(x,g \cdot x) \leq \Lambda \cdot (1+d(x,g \cdot x)) \leq \Lambda \cdot (C+1)$ for all $x\in X$. So,  $g \in G_{X_D}$. \\
If $g \in G_{X_D}$, then $D(x,g \cdot x)\leq C$ for all $x\in X$ and some fixed constant $C$.
Thus, if $(x_n)$ is a Cauchy-Gromov sequence, then $(g\cdot x_n)$ is a Cauchy-Gromov sequence equivalent to $(x_n)$.\\ This is because,
\begin{align*}
    (x_n|g\cdot x_n)_w &= [D(x_n,w) + D(g\cdot x_n,w) - D(x_n, g\cdot x_n)]/2 \\ &\geq [D(x_n,w) + D(x_n,w) - D(x_n,g\cdot x_n) - D(x_n,g\cdot x_n)]/2 \\ &\geq [D(x_n,w) - C] \rightarrow \infty \text{ as } n \rightarrow \infty
\end{align*}
Hence, $g$ fixes $\partial X_D$, implying that $g \in K_D.$ 
\end{proof}

\begin{theorem}
\label{Theorem B}  
Let $X$ be a proper $\mathrm{CAT}(0)$ space with $X_D$ being unbounded, and let $G$ act on $ X$ co-compactly. Then, $G_X=G_{X_D}$.
\end{theorem}
\begin{proof} We start with showing that $G_X$ and $K_X$ coincide.
\begin{claim}
  $G_X=K_X$.
\end{claim}
\begin{proof}[Proof of Claim]
We first show that $K_X\subseteq G_X$. Fix $g\in K_X$ and $x\in X$. As $X_D$ is unbounded, $X$ is unbounded as well, and hence non-compact, and since $G$ acts on $X$ co-compactly, use Corollary 3 in \cite{Ontaneda:boundaryofCAT0} to obtain a constant $C$ such that for each  $y\in X$ there is a geodesic ray $\gamma:[0,\infty] \rightarrow X$ with $\gamma(0) = x$ and $d(y,z) < C$ for some $z \in \mathrm{Im} (\gamma)$. We claim there exists $C'$ not depending on $y$ such that $d(y, g\cdot y)\leq C'$. We have $d(y,g\cdot y) \leq d(y,z)+d(z,g\cdot z)+d(g\cdot z,g\cdot y) \leq 2C+d(z,g \cdot z)$. Moreover, if $d(z,g \cdot z)>d(x,g \cdot x)$ then by the convexity of $\mathrm{CAT}(0)$ distance we have that $d(\gamma(t),g \cdot \gamma(t)) \rightarrow \infty$ which contradicts $ g \in K_X$. Thus, $d(z,g \cdot z)\leq d(x,g \cdot x)$ and hence $d(y,g\cdot y) \leq 2C+d(z,g \cdot z) \leq 2C+d(x,g \cdot x)$. Since $x \in X$ is fixed, $C':= 2C+d(x,g\cdot x)$ does not depend on $y\in X$. So, $g \in G_X$.

Conversely, if $g \in G_X$ then for any $\gamma:[0,\infty] \rightarrow X$, $d(\gamma(t),g \cdot \gamma(t)) < C$ for some constant $C$. Thus, $g$ fixes the point  $\gamma(\infty) \in \partial X$. So $g \in K_X$.
\end{proof}

We claim $K_X = K_D$. We note that in the case that the action of $G$ is properly discontinuous as well, the equality follows by Corollary 1.6 of \cite{baik2024kernel}. 
\\
Observe that $G$ acts co-compactly on $X$, and hence co-boundedly on $X_D$. As $X_D$ is unbounded, so an orbit of $G$ in $X_D$ is unbounded as well. So, by Theorem \ref{classification2}, $\partial X_D$ is nonempty. Thus, by Theorem \ref{theoremL}, $\partial X_D$ embeds as an  Isom $X$-invariant, dense subspace of $\partial X$.
Now, if $g \in K_X$, then $g$ fixes $\partial X$, and hence $g$ also fixes $\partial X_D \hookrightarrow \partial X$. So $g \in K_D$. Conversely, if $g \in K_D$, then $g$ fixes $\partial X_D$, which is dense in $\partial X$. Since  $\partial X$ is hausdorff and $g$ acts as a homeomorphism on $\partial X$, $g$ must fix $\partial X$ as well. Thus, $g \in K_X$.

Therefore, $G_X=K_X=K_D$. Thus, by Lemma \ref{lemma 1} we have that $G_X=G_{X_D}$.
\end{proof}

Next, we examine the structure of the coarse kernel $G_X$ of $X$ when the action of $G$ on $X$ is properly discontinuous but not, in general, co-compact. Specifically, we establish mild conditions concerning the orbit of $G$ in the curtain model $X_D$ that ensure the coarse kernel $G_X$ is virtually cyclic, and likewise finite.

Our key strategy will be to look at the action of the coarse kernel $G_{X_D}$ on the curtain model $X_D$ and apply the classification of actions on hyperbolic spaces, as discussed in Section~\ref{sec2.2}. This will allow us to deduce that the subgroup has either a bounded or quasi-line orbit in $X_D$. We can then exploit the weak acylindricity of the action of $G$ on $X_D$ to obtain the desired result. 

Notice, firstly, that for any group $G$ acting on a metric space $X$ by isometries, $G_X$ is a normal subgroup of $G$. Indeed, if $h\in G_X$, then $d(h\cdot x,x)\leq C$ for all $x\in X$. Therefore, ${d(g^{-1}h g\cdot x,x)= d(h g\cdot x,g\cdot x)\leq C}$  for all $x\in X$, implying that $g^{-1}hg \in G_X$.

\begin{theorem}
\label{Theorem C}
Let $X$ be a co-bounded $\mathrm{CAT}(0)$ space, and let $G $ act on $ X$ properly discontinuously (but not necessarily co-boundedly).  Suppose that $G$ has an unbounded orbit in $X_D$ . Then, $G_X$ is virtually cyclic.
Moreover, if an orbit of $G$ in $X_D$ is not a quasi-line either, then $G_X$ is finite.
\end{theorem}
\begin{proof}
By assumption, the orbit of $G$ in $X_D$ is not bounded. Firstly, we focus on the case it is not a quasi-line either.
\begin{claim}
    If an orbit of $G$ on $X_D$ is not a quasi-line, then the action of $G_{X_D}$ on $X_D$ is bounded.
\end{claim}
\begin{proof}[Proof of Claim]
Since $G_{X_D} \subseteq K_D$, we have that $G_{X_D}$ fixes the boundary of $X_D$. Thus, by Theorem \ref{classification2}, the action of $G_{X_D}$ can be either bounded, horocyclic, or lineal. Indeed, if the action of $G_{X_D}$ is focal or of general type, then $|\partial X_D| = \infty$ and $G_{X_D}$ fixes at most one point in the boundary $\partial X_D$. A contradiction.

If the action is horocyclic, then $|\partial_{X_D}(G_{X_D})| = 1$. Since each element of $\partial X_D$ is fixed by the action of $G_{X_D}$, it follows from Theorem \ref{classification2} that $|\partial X_D| = |\partial_{X_D}(G_{X_D})| = 1$.
But this is impossible since $X$, and hence $X_D$, is co-bounded. Indeed, let $H $ act co-boundedly on $ X_D$. Then, $|\partial_{X_D} H| = |\partial X_D| = 1$, implying that the action of $H$ on $X_D$ is horocyclic. This contradicts Lemma \ref{horocyclicactions}.

If the action is lineal, then $|\partial_{X_D} (G_{X_D})|= 2$. Since each element of $\partial X_D$ is fixed by the action of $G_{X_D}$, it follows from Theorem \ref{classification2} that $|\partial X_D|=|\partial_{X_D} (G_{X_D})|= 2$. Therefore, we have $2=|\partial X_D| \geq |\partial_{X_D} G| \geq |\partial_{X_D} (G_{X_D})|\geq 2$. So, $|\partial_{X_D} G|=2$, implying that the action of $G$ on $X_D$ is lineal. Hence, by Lemma \ref{linealactions}, the orbit of $G$ is a quasi-line, contradicting our assumption.\\
So the action of $G_{X_D}$ on $X_D$ is bounded.
\end{proof}
Fix $x\in X$. By the claim above, $G_{X_D}$ acts boundedly on $X_D$, so there is a constant $C>0$ so that $D(x, h\cdot x) < C$  for all $h\in G_{X_D}$.\\
Since $G$ acts properly discontinuously on $X$, by Theorem \ref{9.51015} we get that the induced action of $G$ on $X_D$ is weakly acylindrical. So let $R> 0$ satisfy that for any $x,y \in X$ with $D(x,y) > R$, only finitely many $h \in G$ have $max\{D(x, h\cdot x), D(y, h\cdot y)\} < C$. 
Now, since $G$ has unbounded orbits in $X_D$, we can always find a point $y$ in the orbit $G\cdot x$ with $D(x,y) > R$. Let $y = g\cdot x$ be such a point. Then, $$D(y, h\cdot y) = D(g\cdot x, h\cdot (g\cdot x))=D(g^{-1} h g\cdot x, x) < C \text{ } \forall h\in G_{X_D}.$$
The last inequality follows because $G_{X_D}\trianglelefteq G$, and thus $g^{-1} h g \in G_{X_D}$.\\
So, $max\{D(x, h\cdot x), D(y, h\cdot y)\} < C$ for all $h\in G_{X_D}$, and $D(x,y) > R$. Hence, $G_{X_D}$ must be finite. So $G_X\subseteq G_{X_D}$ is finite as well.

Now, suppose that the orbit of $G$ in $X_D$ is a quasi-line. From Theorem \ref{9.51015}, we have that $X_D$ is roughly geodesic and $G$ acts weakly acylindrically on $X_D$. So, by Proposition \ref{Theorem D}, $G$ is virtually $\Z$. Hence, $G_X \subseteq G$ is virtually cyclic.
\end{proof}

Finally, using Theorem \ref{Theorem C} above alongside Theorem \ref{Theorem A}, we obtain the following results regarding the algebraic structure of the coarse kernel $G_X$ for a group $G$ that acts geometrically on a $\mathrm{CAT}(0)$ space $X$. Consequently, we prove Corollary \ref{corthmB} characterizing when the coarse kernels $G_X$ and $G_{X_D}$ coincide.

\begin{theorem}
\label{Theorem C'}
Let $X$ be a $\mathrm{CAT}(0)$ space, and let $G $ act on $ X$ geometrically. Then,\\
    (i) $G_X$ is virtually $\Z^n$ for some $n \in \Z_{\geq 0}$.\\
    (ii) If moreover $X_D$ is unbounded, then $G_X$ is virtually cyclic.\\
    (iii) If furthermore $X_D$ is not a quasi-line, then $G_X$ is the largest finite normal subgroup of $G$.
\end{theorem}
\begin{proof}
(i) Since $G$ acts geometrically on $X$, we can apply Theorem \ref{Theorem A} to $Y=X$, yielding $G_X = \bigcup \{C_G(H) \mid |G/H|<\infty \}$. Assume that this infinite union cannot be realized as the centraliser of a single finite index subgroup.\\
Then, we may inductively choose finite index subgroups $H_i$,  $i \in \Z$ as follows:-\\
 Pick $H_1 = G$. Suppose we have chosen finite index subgroups $H_1,H_2,...,H_n$ of $G$ so that we get proper inclusions $H_n \subset H_{n-1} \subset ... \subset H_1$, and ${C_G(H_1) \subset C_G(H_{2}) \subset ... \subset C_G(H_{n})}$. Now, there must exist a finite index subgroup $K_{n+1}$ of $G$ so that $C_G(K_{n+1}) \not \subseteq C_G(H_n)$. 
If not, then we would have $\bigcup \{C_G(H) \mid |G/H|<\infty \} = C_G(H_n)$, contradicting our assumption.
Let $H_{n+1} = K_{n+1} \cap H_n$. Since $H_n$ and $K_{n+1}$ are finite index subgroups of $G$, so is $H_{n+1}$. Observe also that $H_{n+1}$ is a proper subgroup of $H_n$ because ${C_G(K_{n+1}) \not\subseteq C_G(H_n)}$ and $C_G(K_{n+1})\subseteq C_G(H_{n+1})$. In particular, we obtain proper chains $H_{n+1} \subset H_n \subset ... \subset H_1$, and  $C_G(H_1) \subset C_G(H_2) \subset ... \subset C_G(H_{n+1})$. 

Hence, we have an infinite chain $C_G(H_1) \subset C_G(H_2) \subset \cdot \cdot \cdot$ of subgroups of $G$. Moreover, all of these are virtually abelian. Indeed, note that the center $Z(H_i)$ of $H_i$ is an abelian subgroup of $C_G(H_i)$ and we have 
\begin{equation}
\label{e2}
    |C_G(H_i)/Z(H_i)|= |C_G(H_i)/{C_G(H_i)\cap H_i}|= |{C_G(H_i) \cdot H_i}/H_i| \leq |G/H_i| 
\end{equation}
where $|G/H_i|$ is finite as $H_i$ is a finite index subgroup of $G$. \\(The third equality in equation \ref{e2} follows from the second isomorphism theorem applied to the group $N_G(H_i)$ with subgroups $C_G(H_i),H_i \trianglelefteq N_G(H_i)$ ) \\
However, this contradicts the Ascending Chain Condition [\cite{bridsonhaefliger}, Theorem 7.5].

So, $G_X = C_G(H)$ for some finite index subgroup $H$ of $G$. Now, from equation \ref{e2}, we know $Z(H)$ is a finite index subgroup of $C_G(H)$. Also, $Z(H)$ is an abelian subgroup of $G$ and hence, by Corollary 7.6 of \cite{bridsonhaefliger}, must be finitely generated. Thus, $Z(H)$, and hence $C_G(H) = G_X$, is virtually $\Z^n$ for some $n$.

(ii) As $G$ acts on $X$ geometrically, so $X$ is co-bounded. Also, the orbit of $G$ is quasi-dense in $X$ and hence the corresponding orbit in $X_D$ is quasi-dense as well. As $X_D$ is unbounded, we have that the orbit of $G$ in $X_D$ is unbounded. Therefore, by Theorem \ref{Theorem C}, $G_X$ is virtually cyclic.

(iii) Now, if $X_D$ is not a quasi-line then the orbit of $G$ in $X_D$ isn't a quasi-line either. So, by Theorem \ref{Theorem C}, $G_X$ is finite.
Also, $G_X$ is a normal subgroup of $G$.
Moreover, if $N$ is a finite normal subgroup of $G$, then $N$ acts uniformly boundedly on an orbit $G\cdot x_0$, which is quasi-dense in $X$. So $N$ acts uniformly boundedly on $X$, implying that $N \subseteq G_X$. 
Thus, $\bigcup \{N|N \text{ is a finite normal subgroup of }G\} \subseteq G_X =$ finite normal subgroup.
Thus, $G_X$ must be the largest finite normal subgroup of $G$.
\end{proof}

\begin{cor}
\label{cortheoremB}  
Let $G$ be a group acting geometrically on a $\mathrm{CAT}(0)$ space $X$. Then exactly one of the following holds.
    \begin{enumerate}
        \item $X_D$ is unbounded and hence $G_X = G_{X_D}$;
        \item  $X_D$ is bounded, $|G:Z(G)|<\infty$, and $G_X = G = G_{X_D}$; 
        \item $X_D$ is bounded, $Z(G)$ has infinite index in $G$, and $G_X \neq G_{X_D}$.
    \end{enumerate}
\end{cor}

\begin{proof}
    If $X_D$ is unbounded, then by Theorem \ref{Theorem C}, $G_X = G_{X_D}$. Assuming $X_D$ is bounded, we get $G_{X_D}=G$. Therefore, $G_X=G_{X_D}$ if and only if $G=G_X$.
    
    Now, by the proof of Theorem \ref{Theorem C'} part (i), we have that $G_X=C_G(H)$ for some finite index subgroup $H$ of $G$. Thus, if $G=G_X$, then $G=C_G(H)$, and hence $H\subseteq Z(G)$. This implies that $|G:Z(G)|\leq |G:H|<\infty$. Conversely, if $|G:Z(G)|<\infty$, then $G=C_G(Z(G))\subseteq G_X$. Hence, $G_X=G$.
\end{proof}

\bibliographystyle{alpha}
\bibliography{Bibliography}

\begin{thebibliography}{CCMT15}

\bibitem[BH99]{bridsonhaefliger}
Martin~R. Bridson and Andr\'{e} Haefliger.
\newblock {\em Metric spaces of non-positive curvature}, volume 319 of {\em
  Grundlehren der mathematischen Wissenschaften [Fundamental Principles of
  Mathematical Sciences]}.
\newblock Springer-Verlag, Berlin, 1999.

\bibitem[BJ24]{baik2024kernel}
Hyungryul Baik and Wonyong Jang.
\newblock On the kernel of actions on asymptotic cones.
\newblock {\em arXiv preprint arXiv:2402.09969}, 2024.

\bibitem[CCMT15]{caprace:amenable}
Pierre-Emmanuel Caprace, Yves Cornulier, Nicolas Monod, and Romain Tessera.
\newblock Amenable hyperbolic groups.
\newblock {\em J. Eur. Math. Soc. (JEMS)}, 17(11):2903--2947, 2015.

\bibitem[GO07]{Ontaneda:boundaryofCAT0}
Ross Geoghegan and Pedro Ontaneda.
\newblock Boundaries of cocompact proper {${\rm CAT}(0)$} spaces.
\newblock {\em Topology}, 46(2):129--137, 2007.

\bibitem[LV23]{Leitner2023}
Arielle Leitner and Federico Vigolo.
\newblock {\em Coarse Kernels}, pages 119--134.
\newblock Springer Nature Switzerland, Cham, 2023.

\bibitem[PSZ24]{PETYT2024109742}
Harry Petyt, Davide Spriano, and Abdul Zalloum.
\newblock Hyperbolic models for cat(0) spaces.
\newblock {\em Advances in Mathematics}, 450:109742, 2024.

\bibitem[Rua01]{Ruane:actiononboundary}
Kim~E. Ruane.
\newblock Dynamics of the action of a {${\rm CAT}(0)$} group on the boundary.
\newblock {\em Geom. Dedicata}, 84(1-3):81--99, 2001.

\end{thebibliography}

\end{document}